\documentclass[12pt]{article}

\usepackage{amsmath}
\usepackage{amssymb}
\usepackage[mathscr]{eucal}
\usepackage{graphicx}
\usepackage{color}
\usepackage{amsthm}
\usepackage{url}

\usepackage{caption}
\usepackage{subcaption}

\usepackage{amsfonts}
\usepackage{geometry}
\usepackage{tikz}
\usetikzlibrary{arrows.meta}
\usepackage{lineno}

\newtheorem{Theorem}{\sc Theorem}

\newtheorem{Proposition}[Theorem]{\sc Proposition}
\newtheorem{Lemma}[Theorem]{\sc Lemma}

\newtheorem{Remark}[Theorem]{\sc Remark}

\newcommand{\R}{\mbox{{${\mathbb{R}}$}}}

\def\sqr#1#2{{
    \vcenter{
         \vbox{\hrule height.#2pt
               \hbox{\vrule width.#2pt height#1pt \kern#1pt
                     \vrule width.#2pt
               }
               \hrule height.#2pt
         }
    }
}}

\def\lista#1
{{ \itemindent 0.0cm \labelsep .2cm \leftmargin 0.8cm \rightmargin
0.0cm \labelwidth 0.6cm \topsep 0.0mm
\parsep 0.0mm
\itemsep 0.0mm
\begin{list}{}
{ \setlength{\leftmargin}{.8cm} \setlength{\rightmargin}{0.0cm}
\setlength{\parsep}{0.0mm} \setlength{\topsep}{.0mm}
\setlength{\parskip}{.0cm} \setlength{\itemsep}{.0cm} }
{#1}\end{list}} }

\newcounter{theorem}

\begin{document}
%\linenumbers

\title{\bf Finite Element Approximation of a~Hemivariational Inequality
for Steady-State Heat Conduction: Double-Limit Convergence
of Penalization and Discretization}

\vspace{22mm}
{\author{
Piotr Bartman-Szwarc$^{1}$,\
Anna Ochal$^{1}$\footnote{Corresponding author, E-mail : anna.ochal@uj.edu.pl },\\ and\\
Domingo A. Tarzia$^{2,3}$\\[6mm]
{\it \small $^1$ Chair of Optimization and Control, Jagiellonian University in Krakow}\\
{\it\small Lojasiewicza 6, 30348 Krakow, Poland}		\\[6mm]
{\it\small $^{2}$ Departamento de Matem\'atica, FCE, Universidad Austral}\\
		{\it \small Paraguay 1950, S2000FZF Rosario, Argentina}\\
{\it\small $^{3}$ CONICET, Argentina}}}

\date{}
\maketitle
\thispagestyle{empty}

\begin{center}
\textit{Dedicated to Professor Mircea Sofonea on the occasion of his 70th birthday}
\end{center}

\vskip 5mm
\noindent {\small{\bf Abstract.}
In this paper we study the numerical approximation and convergence analysis of a~steady-state heat conduction problem with mixed boundary conditions.
The physical model is governed by a hemivariational inequality depending on a heat transfer parameter $\alpha > 0$.
We consider the finite element approximation of this penalized problem, as well as its corresponding limit problem with a prescribed constant temperature on a part of the boundary.
The main theoretical result establishes the strong convergence of the discrete solutions.
Specifically, we prove the double-limit convergence of the finite element approximations to the limit solution as the mesh size $h$ tends to zero and the penalization parameter $\alpha$ tends to infinity, independently and simultaneously.
We further complement the convergence analysis with an error estimate of optimal order for the discrete limit problem, and we explain why an estimate uniform in the penalization parameter cannot be expected.
The theoretical results are illustrated by numerical simulations on four examples verified by the method of manufactured solutions.
}



\newpage
\section{Introduction}
\setcounter{equation}0

\medskip\noindent
Mathematical models of heat conduction with mixed boundary conditions often lead to complex boundary value problems, especially when the boundary law governing heat exchange on part of the boundary is nonlinear, nonmonotone or multivalued.
Such phenomena arise naturally in thermal engineering and industrial applications, for instance in phase-change processes, in the modelling of contact with a thermostat, and in problems where the heat flux-temperature relationship cannot be described by a smooth, single-valued function.
The mathematical framework for such nonsmooth boundary laws is provided by the theory of hemivariational inequalities, which was originally introduced by Panagiotopoulos \cite{Panagiotopoulos1993} to handle energy functionals that are neither convex nor smooth.
The Clarke subdifferential \cite{Clarke1990} is the key analytical tool that allows the formulation and analysis of these problems.

We consider a bounded domain $\Omega \subset \R^d$ with a regular boundary $\partial\Omega = \Gamma$ which consists of the union of three disjoint parts $\Gamma_1$, $\Gamma_2$ and $\Gamma_3$ with positive measure of $\Gamma_1$ and $\Gamma_3$.
We formulate the following two steady-state heat conduction problems with mixed boundary conditions:
\begin{itemize}
\item[($P_{\infty}$):] $-\Delta u = g$ in $\Omega$,\quad
$u|_{\Gamma_1} = 0$,\quad
$-\frac{\partial u}{\partial n}|_{\Gamma_2} = q$,\quad
$u|_{\Gamma_3} = b$.
\item[($P_{\alpha}$):] $-\Delta u = g$ in $\Omega$,\quad
$u|_{\Gamma_1} = 0$,\quad
$-\frac{\partial u}{\partial n}\big|_{\Gamma_2} = q$,\quad
$-\frac{\partial u}{\partial n}\big|_{\Gamma_3} \in \alpha\,\partial j(u)$.
\end{itemize}
Here $u$ is the temperature in $\Omega$, $n$ is the outward unit normal on $\Gamma$, $g$ is the internal heat source in $\Omega$, $q$ is the prescribed heat flux on $\Gamma_2$, $b$ is the prescribed constant temperature on $\Gamma_3$, $\alpha > 0$ is the heat transfer coefficient, and $\partial j$ denotes the Clarke subdifferential of a superpotential $j\colon \Gamma_3 \times \mathbb{R} \to \mathbb{R}$.
The boundary condition on $\Gamma_3$ in ($P_\alpha$) is a nonlinear Robin condition of hemivariational type, and the prescribed boundary condition on $\Gamma_3$ in ($P_\infty$) is its "Dirichlet limit", i.e., the Dirichlet boundary condition in the limit. Multivalued relations of this type appear in steady-state heat conduction problems, including models of semipermeable membranes and temperature control. Problem ($P_\alpha$) further provides a~prototype for boundary semipermeability models arising in hydraulics, porous media flow, and electrostatics (see \cite{MigOch2004, Panag1985}). The convex case has been addressed in \cite{Barbu1982, DuvautLions1972}.

The specific class of problems considered in this paper arises from a penalty approximation of the Dirichlet boundary condition.
On a portion of the boundary the ideal limit model ($P_\infty$) prescribes the temperature exactly to a value $b$; the regularized model ($P_\alpha$), for each $\alpha>0$, replaces this hard constraint by a Robin-type boundary condition whose flux is given by $\alpha\, \partial j(u)$.
As $\alpha \to +\infty$ one recovers the Dirichlet boundary condition in the limit.
This penalization strategy has a long history in the study of elliptic problems with mixed boundary conditions.
For linear problems, the convergence of solutions $u_\alpha \to u_\infty$ as $\alpha \to \infty$, and related comparison and monotonicity properties, were studied systematically by Gariboldi and Tarzia~\cite{GaTa} in the context of boundary optimal control for Stefan-like problems, and by Tarzia~\cite{Ta2}, and Tabacman and Tarzia~\cite{TaTa} for the steady-state two-phase Stefan problem.
The passage from the Robin boundary condition to the Dirichlet one in the nonlinear, hemivariational setting was addressed by Gariboldi, Mig\'orski, Ochal and Tarzia~\cite{GMOT2021}, who established sufficient conditions on the superpotential $j$ for existence, strong convergence, and comparison of solutions.

The present paper takes the next step by combining the penalization limit $\alpha \to +\infty$ with a finite element discretization (mesh size $h \to 0^+$) and proving that both limits can be taken simultaneously.
Convergence analysis for finite element discretizations of hemivariational inequalities has attracted significant attention. We refer to \cite{HanMigorskiSofonea2014, HanSofonea2002} for a comprehensive treatment and to \cite{MOS2013} for the abstract framework of nonlinear inclusions and hemivariational inequalities.
Double-limit results of a similar flavour, combining a penalization or regularization parameter with a discretization parameter, appear in the context of optimal control problems in \cite{Tarzia2016} and in related elliptic settings in \cite{CarlMotreanu2007}.
To the best of our knowledge, the simultaneous double limit for the finite element approximation of a hemivariational inequality with a penalized Dirichlet condition on part of the boundary has not been addressed in the literature; this is the main contribution of the present work.

The finite element analysis of hemivariational inequalities has advanced substantially in recent years.
Han, Sofonea and Barboteu~\cite{HanSofoneaBarboteu2017} established optimal-order error estimates for elliptic hemivariational inequalities under appropriate solution regularity assumptions.
The comprehensive survey by Han and Sofonea~\cite{HanSofonea2019} covers existence, uniqueness, convergence analysis and numerical simulation for a class of contact problems governed by hemivariational inequalities.
A broad numerical perspective on methods for hemivariational inequalities arising in contact mechanics, including direct optimization, augmented Lagrangian, and primal-dual active set strategies, is provided in \cite{OJB}, where the \textit{conmech} package used in the present work was also first described.
More recently, Han, Feng, Wang and Huang~\cite{HanFengWangHuang2025} unified the convergence analysis for variational-hemivariational inequalities under minimal regularity assumptions, while Wang, Wu, and Han ~\cite{WangWuHan2021} applied the virtual element method in this setting.

To place our work in a proper context, we now indicate how the simultaneous double limit $(h,\alpha) \to (0^+, +\infty)$ goes beyond the single-limit results recalled above, and where the difficulties of the analysis lie.
The difficulties come from the simultaneous limit rather than from either of the two limits separately.
Neither of them can be taken first: the discrete solutions are indexed by both parameters at once, so every estimate on them has to be uniform in $h$ and in $\alpha$ together, and the usual argument, in which one parameter is frozen while the other is sent to its limit, is unavailable.
This is also what separates the present setting from the double-limit results known so far, which are confined to the quadratic potential $j(x,r) = \frac{1}{2}(r-b)^2$ \cite{Tarzia2016}; here $j$ is allowed to be nonconvex and nonsmooth, and the boundary law is governed by the multivalued and nonmonotone Clarke subdifferential.
The penalty term is the source of the difficulty and of the interest alike.
Its integrand is only upper semicontinuous, it carries the factor $\alpha$ that grows without bound, and in the double limit the solutions and the test functions move at the same time, so the term is not controlled by monotonicity and has to be handled by compactness.
The same term is what encodes the Dirichlet condition in the limit: the constraint $u|_{\Gamma_3} = b$ is not imposed on the discrete spaces at any finite $\alpha$ and has to be recovered from the penalty as $\alpha$ grows, which is where the nonmonotone nature of $j$ is felt most directly.

We also establish an error estimate of optimal order for the discrete limit problem.
For the penalized problems no estimate uniform in $\alpha$ can be expected, and we explain why: the smallness condition underlying the existing theory degenerates precisely in the regime $\alpha \to +\infty$, and the discrete problem may genuinely have several solutions at a finite $\alpha$, as our computations show.

The main goal of this paper is to prove that a sequence of discrete solutions to the penalized hemivariational inequality converges strongly to the unique  solution of the limit problem as $(h,\alpha) \to (0^+, +\infty)$; throughout the paper, we use the shorthand $(h,\alpha)\to(0,\infty)$.
Furthermore, we validate our theoretical findings with numerical experiments, computing the convergence rates explicitly.

The rest of the paper is organized as follows.
In Section~2 we establish the preliminary material, including the strong and weak formulations of both problems, the functional spaces, and the fundamental hypotheses.
In Section~3 we introduce the finite element discretization and prove the convergence results.
Section~4 is devoted to numerical simulations that illustrate and quantify the theoretical double-limit convergence.

\section{Preliminaries}
In this section we recall standard notation and preliminary results. In a Banach space~$X$ for a locally Lipschitz function $\varphi\colon X\to \mathbb{R}$ we define (\cite{Clarke1990}):

\begin{itemize}
    \item[] the generalized (Clarke) directional derivative of $\varphi$ at $x\in X$ in the direction $v\in X$
    $$\varphi^0(x;v):=\limsup_{y\to x, \lambda\to 0} \frac{\varphi(y+\lambda v)-\varphi(y)}{\lambda},$$
    \item[] the generalized gradient (subgradient) of $\varphi$ at $x\in X$
    $$\partial \varphi (x):=\{\zeta\in X^*: \varphi^0(x;v)\leq \langle \zeta,v\rangle\quad \forall v\in X\}.$$
\end{itemize}

\noindent
Throughout the paper we use the following notation of spaces and sets:
\begin{itemize}
\item[] $H = L^2(\Omega)$,\quad $\displaystyle (u, v)_H = \int_{\Omega} u v\,dx$\quad  $\forall u, v \in H$,
\item[] $Q = L^2(\Gamma_2)$,\quad  $\displaystyle (u, v)_Q = \int_{\Gamma_2} u v \,dx$\quad  $\forall u, v \in Q$,
\item[] $V = H^1(\Omega)$,\quad  $\|v\|_V^2 = \|v\|_{L^2(\Omega)}^2 + \|\nabla v\|_{L^2(\Omega; \mathbb{R}^d)}^2$\quad  $\forall v \in V$,
\item[] $V_0 = \{v \in V: v|_{\Gamma_1} = 0\}$,\quad  $\|v\|_{V_0} = \|\nabla v\|_{L^2(\Omega; \mathbb{R}^d)}$\quad  $\forall v \in V_0$,
\item[] $K = \{v \in V: v|_{\Gamma_1} = 0, v|_{\Gamma_3} = b\} = \{v \in V_0: v|_{\Gamma_3} = b\}$,
\item[] $K_0 = \{v \in V: v|_{\Gamma_1} = 0, v|_{\Gamma_3} = 0\} = \{v \in V_0: v|_{\Gamma_3} = 0\}$.
\end{itemize}
We remark that from the Poincar\'e inequality the norms $\|\cdot\|_V$ and $\|\cdot\|_{V_0}$ are equivalent.

\noindent
Moreover, we define the bilinear form $a\colon V_0 \times V_0 \to \mathbb{R}$ by
\begin{align*}
a(u, v) := \int_{\Omega} \nabla u \cdot \nabla v \,dx \quad \forall u, v \in V_0.
\end{align*}
The form $a$ is symmetric, continuous and coercive with a constant $m_a > 0$, i.e., $$|a(u,v)|\leq M\|u\|_V\|v\|_V,\quad  a(v, v) = \|v\|_{V_0}^2 \ge m_a \|v\|_V^2 \quad \forall u, v \in V_0 \subset V.$$
The linear form $L\colon V \to \mathbb{R}$ is defined as
\begin{align*}
L(v) := \int_{\Omega} g v\,dx - \int_{\Gamma_2} q \gamma v \,dx \quad \forall v \in V,
\end{align*}
where $\gamma\colon V \to L^2(\Gamma)$ denotes the trace operator on $\Gamma$.
%(In what follows, we write $v$ for the trace of a function $v \in V$ on the boundary $\Gamma$).
If $g\in L^2(\Omega)$ and $q \in L^2(\Gamma_2)$, then
the form $L$ is continuous, i.e.,  from the Cauchy-Schwarz inequality we have
\begin{align*}
|L(v)| \le (\|g\|_{L^2(\Omega)} + \|q\|_{L^2(\Gamma_2)} \|\gamma\|) \|v\|_V = c_L \|v\|_V  \quad \forall v\in V \ \text{with\ } c_L>0.
\end{align*}

\noindent
In a standard way, we obtain the following weak formulation of Problems ($P_{\infty}$) and ($P_{\alpha}$), respectively:
\begin{itemize}
\item[($S_{\infty}$):] find $u_{\infty} \in K$ \ such that \ $a(u_{\infty}, v) = L(v)$ \quad for all $v \in K_0$\\
or equivalently\\
find $u_{\infty} \in K$ \ such that \ $a(u_{\infty}, v - u_{\infty}) = L(v - u_{\infty})$ \quad for all $v \in K$,
\item[($S_{\alpha}$):] find $u_{\alpha} \in V_0$ \ such that \  $\displaystyle a(u_{\alpha}, v) + \alpha\int_{\Gamma_3} j^0(\gamma u_{\alpha}; \gamma v) \,d\Gamma \ge L(v)$ \quad for all $v \in V_0$.
\end{itemize}
Problem ($S_{\infty}$) corresponds to a variational equality of the first kind, whereas ($S_{\alpha}$) takes the form of a~hemivariational inequality.

We admit the following hypothesis on the data.\\[2mm]
\noindent
($H_0$): $g \in L^2(\Omega)$, $q \in L^2(\Gamma_2)$, $b \in \mathbb{R}$. \\[2mm]
$H(j)$: $j\colon \Gamma_3 \times \mathbb{R} \rightarrow \mathbb{R}$ is such that:
\begin{itemize}
\item[(a)] $j(\cdot, r)$ is measurable for all $r \in \mathbb{R}$,
\item[(b)] $j(x, \cdot)$ is locally Lipschitz for a.e. $x \in \Gamma_3$,
\item[(c)] there exist $c_0, c_1 \ge 0$ such that $|\partial j(x, r)| \le c_0 + c_1|r|$ for all $r \in \mathbb{R}$, a.e. $x \in \Gamma_3$,
\item[(d)] $j^0(x, r; b-r) \le 0$ for all $r \in \mathbb{R}$, a.e. $x \in \Gamma_3$,
\item[(e)] $j^0(x, r; b-r) = 0 \iff r = b$ for a.e. $x \in \Gamma_3$.
\end{itemize}
Note that the constant $b$ in $H(j)(d)-(e)$ is the same as in the boundary condition on $\Gamma_3$ in ($P_{\infty}$). Hypothesis $H(j)(e)$, in conjunction with $H(j)(d)$, plays a crucial role in identifying the boundary behavior of the limit of the penalized solutions (see the proofs of Lemma~\ref{L1} and Theorem~\ref{Tall} below).
The existence results for elliptic variational equations and boundary hemivariational inequalities are well known and specified in the following remark.

\begin{Remark}\label{Rexist}
For every $\alpha>0$ Problem ($S_{\alpha}$) has a solution $u_{\alpha} \in V_0$ provided hypothesis ($H_0$) and $H(j)(a)-(c)$ hold (cf. \cite[Theorem~4]{GMOT2021}).
Problem ($S_{\infty}$) has the unique solution $u_{\infty} \in K$ under assumption ($H_0$) (cf. \cite[Corollary~2.102]{CarlMotreanu2007} and \cite[Teorema~1]{Ta1}).
\end{Remark}

Now we prove a useful property of the superpotential $j$.

\begin{Lemma} \label{L_j.1}
If assumptions $H(j)(a)-(c)$ hold, then
for every $w_k \to w$ weakly in $V$ and $v_k \to v$ weakly in $V$, as $k \rightarrow \overline{k} \in \overline{\mathbb{R}}^m$, we have
\begin{align}\label{j.1}
\limsup_{k \rightarrow \overline{k}} \int_{\Gamma_3} j^0(\gamma w_k; \gamma v_k)\,d\Gamma \le \int_{\Gamma_3} j^0(\gamma w; \gamma v)\,d\Gamma.
\end{align}
\end{Lemma}
\begin{proof}
First, we observe that by the compactness of the trace operator, we have $w_k|_{\Gamma_3}~\rightarrow~w|_{\Gamma_3}$ and $v_k|_{\Gamma_3}~\rightarrow~v|_{\Gamma_3}$ in $L^2(\Gamma_3)$, as $k~\rightarrow~\overline{k}$.
Passing to subsequences, we may suppose that
\begin{align*}
w_k(x) \rightarrow w(x) \text{ and } v_k(x)~\rightarrow~v(x) \text{ for a.e. } x \in \Gamma_3,
\end{align*}
and there exist $h_w, h_v \in L^2(\Gamma_3)$ (cf. \cite[Theorem~4.9]{Brezis2011}) such that
\begin{align*}
|w_k(x)| \le h_w(x) \text{ and } |v_k(x)| \le h_v(x) \text{ a.e. } x \in \Gamma_3.
\end{align*}
Using the upper semicontinuity of the function $\mathbb{R} \times \mathbb{R} \ni (r, s) \mapsto j^0(x, r; s) \in \mathbb{R}$ for a.e. $x \in \Gamma_3$ (see \cite[Proposition 3.23 (ii)]{MOS2013}), we get
\begin{align*}
\limsup_{k \rightarrow \overline{k}} j^0(x, w_k(x); v_k(x)) \le j^0(x, w(x); v(x)) \quad \text{a.e. } x \in \Gamma_3.
\end{align*}

\noindent
Next, taking into account the estimate (\cite[Proposition 3.23 (iii)]{MOS2013})
\begin{align*}
j^0(x, w_k(x); v_k(x)) &= \max \{ \langle \psi_k(x), v_k(x) \rangle : \psi_k(x) \in \partial j(x, w_k(x)) \} \\
&\le (c_0 + c_1|w_k(x)|)|v_k(x)| \\
&\le h(x) \quad \text{a.e. } x \in \Gamma_3
\end{align*}
with $h(\cdot) \in L^1(\Gamma_3)$ given by $h(x) = (c_0 + c_1 h_w(x)) h_v(x)$, by the Fatou lemma \cite[Theorem~1.64]{MOS2013} we obtain the desired inequality \eqref{j.1}.
\end{proof}
\noindent
We remark that in the main result (in Section~3) we use the property \eqref{j.1} for $k = h \to 0 = \overline{k}$, $k = \alpha \to \infty = \overline{k}$ ($m = 1$), and $k = (h,\alpha) \to (0,\infty) = \overline{k}$ ($m = 2$). This gives the useful convergence
which allows us to pass to the limit in the hemivariational term.

\section{Discretization by Finite Element Method}

Now we consider the finite element method. Let $\Omega \subset \mathbb{R}^d$ be a polygonal domain with a regular triangulation $\mathcal{T}^h$ consisting of affine-equivalent Lagrange
finite elements of class $C^0$ and of degree one (referred to as triangles of
type~(1) in the terminology of~\cite{Ciarlet2002}), i.e., the $\mathbb{P}_1$ elements
described below.
Here the parameter $h$ of the finite element approximation goes to zero (see \cite{BrennerScott2008}).
We can approximate the sets $V$, $V_0$, $K$ and $K_0$, respectively, by:

\begin{itemize}
\item[] $V^h = \{v^h \in C^0(\overline{\Omega}): v^h|_T \in \mathbb{P}_1(T) \quad \forall T \in \mathcal{T}^h\} \subset V$,
\item[] $V_0^h = \{v^h \in V^h: v^h|_{\Gamma_1} = 0\} \subset V_0$,
\item[] $K^h = \{v^h \in V_0^h: v^h|_{\Gamma_3} = b\} \subset V_0^h$,
\item[] $K_0^h = \{v^h \in V_0^h: v^h|_{\Gamma_3} = 0\}  \subset V_0^h$,
\end{itemize}
where $\mathbb{P}_1$ is the set of the polynomials of degree 1.
Throughout the rest of the paper we assume that the triangulation resolves the partition of the boundary, that is, $\Gamma_1$ and $\Gamma_3$ are unions of edges of $\mathcal{T}^h$.
Since $b$ is a constant, every $v^h \in K^h$ then satisfies $v^h|_{\Gamma_1} = 0$ and $v^h|_{\Gamma_3} = b$ pointwise, so that $K^h \subset K$, $K_0^h \subset K_0$, and the approximation is conforming.

It is clear that such a family $\{V^h\}$ of finite-dimensional vector spaces is an interior approximation to $V$, and a family $\{K^h\}$ of closed convex subsets of $V^h$ approximates $K$ in the sense of Mosco (see \cite[Sec.~4]{Ciarlet2002}):
\begin{equation}
\left\{
\begin{aligned}\label{!}
 &\text{for all } v \in K \ \text{there exists}\  v^h \in K^h \ \text{ such that } v^h \longrightarrow v \ \text{strongly in } V, \text{ as } h \rightarrow 0,
\\
 %\label{!!}
 & \text{if } v^h \in K^h \text{ and } v^h \rightarrow v \text{ weakly in } V, \text{ then } v \in K.
\end{aligned}
\right.
\end{equation}

Let $\Pi^h\colon C^0(\overline{\Omega}) \rightarrow V^h$ be the linear interpolation operator, i.e., $\Pi^h(v) \in V^h$ for all $v \in C^0(\overline{\Omega})$ and $\Pi^h(v(\Lambda)) = v(\Lambda)$ for all $\Lambda \in \Sigma^h$ = $\{ \Lambda \in \overline{\Omega} : \Lambda$ is a vertex of $T \in \mathcal{T}^h\}$.
Then, there exists a constant $c_p > 0$, independent of the parameter $h$, such that (cf. \cite{BrennerScott2008}):
\begin{itemize}
\item[] $\|v - \Pi^h(v)\|_{L^2(\Omega)} \le c_p h^r \|v\|_{H^r(\Omega)} \quad \forall v \in H^r(\Omega)$, $1 < r \le 2$,
\item[] $\|v - \Pi^h(v)\|_V \le c_p h^{r-1} \|v\|_{H^r(\Omega)} \quad \forall v \in H^r(\Omega)$, $1 < r \le 2$.
\end{itemize}

We consider the corresponding discrete problems:
\begin{itemize}
\item[($S^h_{\infty}$):]
find $u^h_{\infty} \in K^h$ \ such that \ $a(u^h_{\infty}, v^h - u^h_{\infty}) = L(v^h - u^h_{\infty})$ \quad for all $v^h \in K^h$,
\item[($S^h_{\alpha}$):] find $u^h_{\alpha} \in V^h_0$ \ such that \  $\displaystyle a(u^h_{\alpha}, v^h) + \alpha\int_{\Gamma_3} j^0(\gamma u^h_{\alpha}; \gamma v^h) \,d\Gamma \ge L(v^h)$ \quad for all $v^h \in V^h_0$.
\end{itemize}

\noindent We keep the assumptions from Section~2. Using similar arguments, as in the classical problems, guarantees that there exist the unique solution to Problem ($S^h_{\infty}$) and a solution to Problem ($S^h_{\alpha}$).

\

Our aim is to study the convergence of a sequence of solutions and obtain a complete commutative diagram, as in Figure~\ref{fig:diagram}. The diagram relates the boundary hemivariational inequalities ($S_{\alpha}$) and variational equality ($S_{\infty}$) with the discrete boundary hemivariational inequalities ($S^h_{\alpha}$) and the discrete variational equality ($S^h_{\infty}$). It is obtained by taking the limits
when $h\to 0$ (for fixed $\alpha>0$ and for $\alpha=\infty$), $\alpha\to\infty$ (for fixed $h>0$) and $(h,\alpha)\to (0,\infty)$.

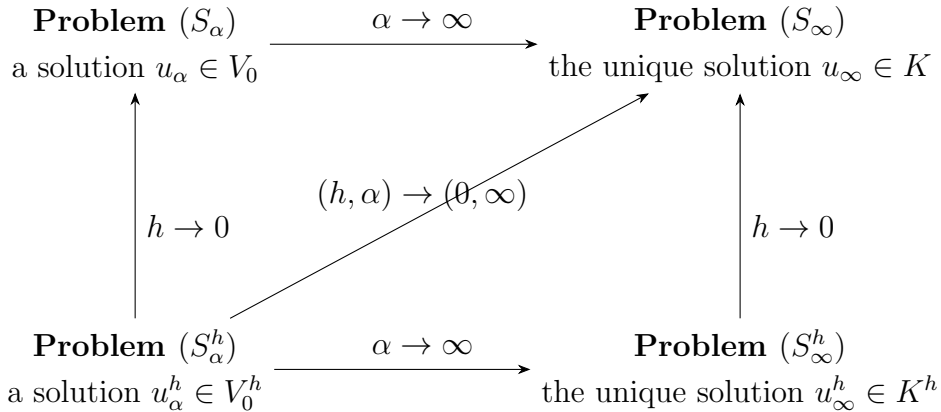
\begin{figure}[ht!]
\begin{center}
\begin{tikzpicture}[>=Stealth, every node/.style={align=center}]

\node (S0) at (0,2.5) {\textbf{Problem $(S_\alpha)$}\\[1mm] a solution $u_\alpha \in V_0$};
\node (uh1) at (8,2.5) {\textbf{Problem $(S_\infty)$}\\[1mm] the unique solution $u_\infty \in K$};

\draw[->] (S0) -- (uh1);
\node at (3.8,2.8) {$\alpha \to \infty$};

\node (S3) at (0,-1.8) {\textbf{Problem $(S^h_\alpha)$}\\[1mm] a solution $u^h_\alpha \in V^h_0$};
\node (uh3) at (8,-1.8) {\textbf{Problem $(S^h_\infty)$}\\[1mm] the unique solution $u^h_\infty \in K^h$};

\draw[->] (S3) -- (uh3);
\node at (3.8,-1.5) {$\alpha \to \infty$};

\draw[->] (S3) -- (S0);
\node at (0.7,0.1) {$h \to 0$};
\draw[->] (uh3) -- (uh1);
\node at (8.7,0.1) {$h \to 0$};

\draw[->] (S3) -- (uh1);
\node at (3.8,0.5) {$(h,\alpha) \to (0,\infty)$};

\end{tikzpicture}
 \caption{Relationships between convergences.}
    \label{fig:diagram}
\end{center}
\end{figure}

\begin{Remark}
Under assumptions ($H_0$) and $H(j)$, if $u_\alpha$ is a solution to ($S_\alpha$) and $u_\infty$ is the~unique solution to ($S_\infty$), then
$$ u_\alpha \to u_\infty \quad \text{strongly in } V, \text{ as } \alpha \to \infty.$$
This convergence result is established in~\cite[Theorem~7]{GMOT2021}, therefore, we omit its proof.
\end{Remark}

In what follows, we focus on the remaining convergence results: the convergence of the sequence $\{u_\infty^h\}$, when $h\to 0$ (Lemma~\ref{L1}), and the convergence of the sequence $\{u_\alpha^h\}$ in three regimes, namely as $h \to 0$, $\alpha \to \infty$, and $(h,\alpha) \to (0,\infty)$ (Theorem~\ref{Tall}).

\

The following lemma presents convergence result for the elliptic variational equality.

\begin{Lemma} %$[h \rightarrow 0^+, \infty]$
\label{L1} Assume that ($H_0$) and $H(j)$ hold.
If $u_{\infty}^h$ is the unique solution to $(S_{\infty}^h)$ and $u_{\infty}$ is the unique solution to $(S_{\infty})$, then
\[
u_{\infty}^h \to u_{\infty} \quad \text{strongly in } V, \text{ as } h \to 0.
\]
\end{Lemma}

\begin{proof}
Let $u_{\infty}^h$ be the unique solution to ($S_{\infty}^h$), i.e.,
\begin{equation}\label{1.1}
u_{\infty}^h \in K^h : \quad   a(u_{\infty}^h, v^h - u_{\infty}^h) = L(v^h - u_{\infty}^h) \quad \forall v^h \in K^h.
\end{equation}
We show that $u_{\infty}^h$ is bounded in $V$ for every $h>0$.
Indeed, by \eqref{!} for $u_{\infty} \in K$ there exists $v^h \in K^h$\  such that \  $v^h \rightarrow u_{\infty}$ strongly in $V$.
Hence, $v^h$ is bounded in $V$ as a strongly convergent sequence, i.e., $\|v^h\|_V \le C_1$ with a constant $C_1 > 0$ independent of $h>0$.
Using \eqref{1.1} and the facts that $a$ is coercive and continuous, and $L$ is continuous, we obtain:
\begin{align*}
m_a \|v^h - u_{\infty}^h\|_V^2 &\le a(v^h - u_{\infty}^h, v^h - u_{\infty}^h) \\
&= a(v^h, v^h - u_{\infty}^h) - a(u_{\infty}^h, v^h - u_{\infty}^h) \\
&= a(v^h, v^h - u_{\infty}^h) - L(v^h - u_{\infty}^h) \\
&\le (M\|v^h\|_V + \|L\|_{V^*}) \|v^h - u_{\infty}^h\|_V,
\end{align*}
and subsequently
\begin{align*}
\|u_{\infty}^h\|_V &\le \|v^h - u_{\infty}^h\|_V + \|v^h\|_V \\
&\le \frac{1}{m_a} (M C_1 + \|L\|_{V^*}) + C_1 \\
&=: C_2,
\end{align*}
which means that $\{u_{\infty}^h\}_h$ is bounded in $V$.
Since $V$ is a reflexive space, then for a subsequence, indexed in the same way, we have
\begin{align*}
u_{\infty}^h \longrightarrow \chi \quad \text{weakly in } V \text{ with } \chi \in V.
\end{align*}
Moreover, by the compactness of the trace operator, we have $u_{\infty}^h|_{\Gamma} \rightarrow \chi|_{\Gamma}$ in $L^2(\Gamma)$, as $h \rightarrow 0$.
From the fact that $u_{\infty}^h|_{\Gamma_1} = 0$ and $u_{\infty}^h|_{\Gamma_3} = b$, we get $\chi|_{\Gamma_1} = 0$ and $\chi|_{\Gamma_3} = b$, respectively, so $\chi \in V_0$ and $\chi \in K$.
Now we show that $\chi \in K$ is a solution to ($S_{\infty}$).
To this end, let $v \in K$ and $v^h \in K^h$ such that $v^h \rightarrow v$ strongly in $V$, as $h \rightarrow 0$ (cf. \eqref{!}).
We take the upper limit of \eqref{1.1}
\begin{align*}
\lim_{h \rightarrow 0} L(v^h - u_{\infty}^h) = \lim_{h \rightarrow 0} a(u_{\infty}^h, v^h) + \limsup_{h \rightarrow 0} (-a(u_{\infty}^h, u_{\infty}^h)) = a(\chi, v) - \liminf_{h \rightarrow 0} a(u_{\infty}^h, u_{\infty}^h).
\end{align*}
Using the weak lower semicontinuity of the functional $V \ni v \mapsto a(v, v) \in \mathbb{R}$ we have
\begin{equation}\label{1.4}
L(v - \chi) \le a(\chi, v) - a(\chi, \chi) = a(\chi, v - \chi) \quad \forall v \in K.
\end{equation}
Let $w \in K_0$.
Taking  as a test function $v = \chi \pm w \in K$ in \eqref{1.4}, we conclude $a(\chi, w) = L(w)$ for all $w \in K_0$, which is equivalent to
\begin{align*}
a(\chi, v - \chi) = L(v - \chi) \quad \forall v \in K.
\end{align*}
Hence, $\chi \in K$ is a solution to ($S_{\infty}$), and by uniqueness of the solution to ($S_{\infty}$) we have $\chi = u_{\infty}$, which gives
$$u_{\infty}^h \rightarrow \chi = u_{\infty} \text{ weakly in } V, \text{ as } h \rightarrow 0.$$

Finally, we show that $u_{\infty}^h \rightarrow u_{\infty}$ strongly in $V$, as $h \rightarrow 0$.
We take $v^h \in K^h$ such that $v^h \rightarrow u_{\infty}$ strongly in $V$, as $h \rightarrow 0$ (cf.\ \eqref{!}).
Using the coercivity of the form $a$ and \eqref{1.1}, we have
\begin{align*}
m_a \|u_{\infty}^h - v^h\|_V^2 &\le a(u_{\infty}^h - v^h, u_{\infty}^h - v^h) \\
&= a(u_{\infty}^h, u_{\infty}^h - v^h) - a(v^h, u_{\infty}^h - v^h) \\
&= L(u_{\infty}^h - v^h) - a(v^h, u_{\infty}^h - v^h).
\end{align*}
Since $u_{\infty}^h - v^h \rightarrow u_{\infty} - u_{\infty} = 0$ weakly in $V$ and $v^h \rightarrow u_\infty$ strongly in $V$, we deduce $L(u_{\infty}^h - v^h) \rightarrow 0$ and $a(v^h, u_{\infty}^h - v^h) \rightarrow 0$, as $h \rightarrow 0$.
Thus
\begin{align*}
    \|u_{\infty}^h - v^h\|_V \rightarrow 0, \text{ as } h \rightarrow 0,
\end{align*}
and consequently
\begin{align*}
0 \le \|u_{\infty}^h - u_{\infty}\|_V \le \|u_{\infty}^h - v^h\|_V + \|v^h - u_{\infty}\|_V \rightarrow 0, \quad \text{as } h \rightarrow 0,
\end{align*}
which implies $u_{\infty}^h \to u_{\infty}$ strongly in $V$, as $h \rightarrow 0$.
\end{proof}

Since the approximation is conforming, the convergence stated in Lemma~\ref{L1} can be complemented by an error estimate of optimal order.

\begin{Proposition}\label{Pcea}
Assume that ($H_0$) holds.
If $u_{\infty} \in K$ and $u_{\infty}^h \in K^h$ are the unique solutions to ($S_{\infty}$) and ($S_{\infty}^h$), respectively, then $u_{\infty}^h$ is the best approximation of $u_{\infty}$ in $K^h$ with respect to $\|\cdot\|_{V_0}$, that is,
\begin{align}\label{cea}
\|u_{\infty} - u_{\infty}^h\|_{V_0} = \min_{v^h \in K^h} \|u_{\infty} - v^h\|_{V_0} .
\end{align}
If, in addition, $u_{\infty} \in H^r(\Omega)$ with $\max\{1, d/2\} < r \le 2$, then
\begin{align}\label{cea2}
\|u_{\infty} - u_{\infty}^h\|_V \le \frac{c_p}{\sqrt{m_a}}\, h^{r-1}\, \|u_{\infty}\|_{H^r(\Omega)},
\end{align}
that is, the convergence of Lemma~\ref{L1} is of optimal order.
\end{Proposition}

\begin{proof}
For $v^h \in K^h$ we have $v^h - u_{\infty}^h \in K_0^h \subset K_0$, so ($S_{\infty}$) and ($S_{\infty}^h$) both give $a(\cdot,\, v^h - u_{\infty}^h) = L(v^h - u_{\infty}^h)$ and, subtracting,
\begin{align}\label{galerkin}
a(u_{\infty} - u_{\infty}^h,\, v^h - u_{\infty}^h) = 0 \qquad \forall\, v^h \in K^h .
\end{align}
Since $a$ is the inner product of $V_0$ and $K^h$ is an affine subset of $V_0$, \eqref{galerkin} states precisely that $u_{\infty}^h$ is the orthogonal projection of $u_{\infty}$ onto $K^h$, which is \eqref{cea}.
For \eqref{cea2}, $r > d/2$ gives $H^r(\Omega) \hookrightarrow C^0(\overline{\Omega})$, so $\Pi^h(u_{\infty})$ is well defined; being affine on each edge and equal to $0$ and to $b$ at every vertex on $\Gamma_1$ and $\Gamma_3$, it belongs to $K^h$.
Taking $v^h = \Pi^h(u_{\infty})$ in \eqref{cea} and using $\|\cdot\|_{V_0} \le \|\cdot\|_V$ on the right-hand side and $\|\cdot\|_V \le m_a^{-1/2}\|\cdot\|_{V_0}$ on the left gives the assertion.
\end{proof}

The analysis of the remaining limits requires higher regularity of solutions than $H^1(\Omega)$. In general, a solution to a mixed elliptic boundary value problem belongs to $H^r(\Omega)$ with $1 < r < \tfrac{3}{2} - \varepsilon$ ($\varepsilon > 0$) (cf.~\cite{LanCap2008, Shamir1968}).
In what follows, we present the main convergence results.

\begin{Theorem} %$[S_{\alpha}^h \rightarrow S_{\alpha}$, $S_{\alpha}^h \rightarrow S_{\infty}^h$, $S_{\alpha}^h \rightarrow S_{\infty}]$
\label{Tall}
Let ($H_0$) and $H(j)$ hold.
We assume the unique solution to ($S_\infty$) has the regularity $u_\infty \in H^r(\Omega)$, $1 < r \leq 2$.
If $u_{\alpha}^h$ is a solution to ($S_{\alpha}^h$), then:
\begin{enumerate}
\item $u_{\alpha}^h \to u_{\alpha}$ strongly in $V$, as $h \to 0$ (with $\alpha > 0$ fixed), where $u_{\alpha}$ is a solution to ($S_{\alpha}$),
\item $u_{\alpha}^h \to u_{\infty}^h$ strongly in $V$, as $\alpha \to \infty$ (with $h>0$ fixed), where $u_{\infty}^h$ is the unique solution to ($S_{\infty}^h$),
\item $u_{\alpha}^h \to u_{\infty}$ strongly in $V$, as $(h, \alpha) \to (0, \infty)$, where $u_{\infty}$ is the unique solution to ($S_{\infty}$).
\end{enumerate}
\end{Theorem}

\begin{proof}
Let $h>0$ and $\alpha>0$.
Let $u_{\alpha}^h \in V_0^h$ be a solution to ($S_{\alpha}^h$), i.e.,
\begin{equation}\label{2.1}
a(u_{\alpha}^h, v^h) + \alpha\int_{\Gamma_3} j^0(\gamma u_{\alpha}^h; \gamma v^h) \,d\Gamma \ge L(v^h) \quad \forall v^h \in V_0^h.
\end{equation}

Let $u_{\infty} \in K \cap H^r(\Omega)$ ($1 < r \le 2$) be the unique solution to ($S_{\infty}$).
Considering the linear interpolation operator $\Pi^h\colon V_0 \to V_0^h$, we know that
\begin{align}
    \label{2.2}
    \Pi^h(u_{\infty}) \longrightarrow u_{\infty} \text{ strongly in } V, \text{ as } h \rightarrow 0 \text{ and } \Pi^h(u_{\infty})|_{\Gamma_3} = b,
\end{align}
so we have
$$\Pi^h(u_{\infty}) \in K^h \subset V_0^h \text{ and } \|\Pi^h(u_{\infty})\|_V \le \tilde{C}, \text{ with } \tilde{C} > 0 \text{ independent of } h.$$

We claim the boundedness of $u_{\alpha}^h$ in $V$.
Indeed, we denote $w^h = \Pi^h(u_\infty)$, take $v^h = w^h - u_{\alpha}^h \in V_0^h$ in \eqref{2.1}, and obtain
\begin{align*}
a(u_{\alpha}^h, w^h - u_{\alpha}^h) + \alpha\int_{\Gamma_3} j^0(\gamma u_{\alpha}^h; \gamma (w^h - u_{\alpha}^h))\, d\Gamma \ge L(w^h - u_{\alpha}^h).
\end{align*}
Next, by \eqref{2.2} and $H(j)(d)$ we have
\begin{align}
\label{2.3}
a(u_{\alpha}^h, u_{\alpha}^h - w^h) \le \alpha\int_{\Gamma_3} j^0(\gamma u_{\alpha}^h; b- \gamma u_{\alpha}^h))\, d\Gamma + L(u_{\alpha}^h - w^h)
\le L(u_{\alpha}^h - w^h).
\end{align}
Then, by coercivity of $a$ together with continuity of $a$ and $L$, and \eqref{2.3}, we get
\begin{align*}
m_a \|u_{\alpha}^h - w^h\|_V^2 &\le a(u_{\alpha}^h - w^h, u_{\alpha}^h - w^h) \\
&= a(u_{\alpha}^h, u_{\alpha}^h - w^h) - a(w^h, u_{\alpha}^h - w^h) \\
&\le (\|L\|_{V^*} + M\|w^h\|_V) \|u_{\alpha}^h - w^h\|_V.
\end{align*}
Since $\|w^h\|_V=\|\Pi^h(u_{\infty})\|_V\leq \tilde{C}$, we conclude
\begin{equation}
\label{2.4}
\|u_{\alpha}^h - w^h\|_V \le \frac{1}{m_a} (\|L\|_{V^*} + M\tilde{C}) =: \overline{C}
\end{equation}
and
\begin{align*}
\|u_{\alpha}^h\|_V \le \|u_{\alpha}^h - w^h\|_V + \|w^h\|_V \le \overline{C} + \tilde{C} =: C_0.
\end{align*}
Hence,
\begin{align}
\label{2.5}
\|u_{\alpha}^h\|_V \le C_0 \quad \text{ for all } h>0,\ \alpha>0,
\end{align}
where $C_0>0$ is independent of $h>0$ and $\alpha>0$.

\

Then, for subsequences, if necessary, indexed in the same way, we have:
\begin{enumerate}
\item $\exists\ \zeta_{\alpha} \in V$ such that $u_{\alpha}^h \rightarrow \zeta_{\alpha}$ weakly in $V$, as $h \rightarrow 0,$ for every fixed $\alpha>0$,
\item $\exists\ \xi^h \in V^h$ such that $u_{\alpha}^h \rightarrow \xi^h$ weakly in $V$, as $\alpha \rightarrow \infty$ for every fixed $h>0$,
\item $\exists\ \eta \in V$ such that $u_{\alpha}^h \rightarrow \eta$ weakly in $V$, as $(h, \alpha) \rightarrow (0, \infty)$.
\end{enumerate}
We observe that since the trace $\gamma: V \rightarrow L^2(\Gamma)$ is compact and $u_{\alpha}^h|_{\Gamma_1} = 0$, as an element of $V_0^h \subset V_0$, we have:
\begin{itemize}
\item[a)] $\zeta_{\alpha} \in V_0$ for every $\alpha > 0$,
\item[b)] $\xi^h \in V_0^h$ for every $h>0$,
\item[c)] $\eta \in V_0$.
\end{itemize}
We need to prove that:
\begin{enumerate}
\item[(1.)] $\zeta_{\alpha} \in V_0$ is a solution to ($S_{\alpha}$) and $\|u_{\alpha}^h - \zeta_{\alpha}\|_V \rightarrow 0$, as $h \rightarrow 0$, for every fixed $\alpha>0$,
\item[(2.)] $\xi^h \in V_0^h$ is a solution to ($S_{\infty}^h$), that is $\xi^h = u_{\infty}^h$ by uniqueness, and $\|u_{\alpha}^h - \xi^h\|_V \rightarrow 0$, as $\alpha \rightarrow \infty$, for every fixed $h>0$,
\item[(3.)] $\eta \in V_0$ is a solution to ($S_{\infty}$), that is $\eta = u_{\infty}$ by uniqueness, and $\|u_{\alpha}^h - \eta\|_V \rightarrow 0$, as $(h, \alpha) \rightarrow (0, \infty)$.
\end{enumerate}

\noindent
We start with the proof of (1.). We need to show that for every fixed $\alpha>0$:
\begin{enumerate}
\item[(i)] $\zeta_{\alpha} \in V_0$ is a solution to ($S_{\alpha}$),
\item[(ii)] $\|u_{\alpha}^h - \zeta_{\alpha}\|_V \rightarrow 0,$ as $h \rightarrow 0$.
\end{enumerate}
Let $v \in V_0$.
Since $V_0^h$ is an interior approximation to $V_0$, there exists a sequence $v^h \in V_0^h$ such that $v^h \rightarrow v$ strongly in $V$, as $h \rightarrow 0$.
Taking such $v^h \in V_0^h$ in \eqref{2.1}, the upper limit when $h \rightarrow 0$ and applying the property \eqref{j.1} with $v_k=v^h, w_k = u_{\alpha}^h \rightarrow \zeta_{\alpha} = w$, as $k = h \rightarrow 0$, we have
\begin{align*}
a(\zeta_{\alpha}, v) + \alpha\int_{\Gamma_3} j^0(\gamma \zeta_{\alpha}; \gamma v)\,d\Gamma \ge L(v) \quad \forall v \in V_0,
\end{align*}
which means that $\zeta_{\alpha}$ is a solution to ($S_{\alpha}$).
Hence, we may assume that $\zeta_{\alpha} = u_{\alpha}$ is a solution to ($S_{\alpha}$).

Next, again by the fact that $V_0^h$ is an interior approximation to $V_0$, we consider a~sequence $\zeta_{\alpha}^h \in V_0^h$ such that $\zeta_{\alpha}^h \rightarrow \zeta_{\alpha}$ strongly in $V$, as $h \rightarrow 0$.
By the coercivity of $a$, we get
\begin{align}
\label{2.6}
\begin{split}
m_a \|\zeta_{\alpha} - u_{\alpha}^h\|_V^2 &\le a(\zeta_{\alpha} - u_{\alpha}^h, \zeta_{\alpha} - u_{\alpha}^h) \\
&= a(\zeta_{\alpha}, \zeta_{\alpha} - u_{\alpha}^h)
- a(u_{\alpha}^h, \zeta_{\alpha} - \zeta_{\alpha}^h)
- a(u_{\alpha}^h, \zeta_{\alpha}^h - u_{\alpha}^h).
\end{split}
\end{align}
The first two terms tend to $0$. Indeed,
\begin{align}
\label{2.7}
&a(\zeta_{\alpha}, \zeta_{\alpha} - u_{\alpha}^h) \rightarrow 0, \text{ as } h \rightarrow 0 \text{ since } u_{\alpha}^h \rightarrow \zeta_{\alpha} \text{ weakly in } V,\\
\label{2.8}
&|a(u_{\alpha}^h, \zeta_{\alpha} - \zeta_{\alpha}^h)| \le M \|u_{\alpha}^h\|_V \|\zeta_{\alpha} - \zeta_{\alpha}^h\|_V \rightarrow 0.
\end{align}
Taking $v^h = \zeta_{\alpha}^h - u_{\alpha}^h \in V_0^h$ in \eqref{2.1}, we get
\begin{equation*}
-a(u_{\alpha}^h, \zeta_{\alpha}^h - u_{\alpha}^h) \le \alpha\int_{\Gamma_3} j^0(\gamma u_{\alpha}^h; \gamma (\zeta_{\alpha}^h - u_{\alpha}^h))\,d\Gamma - L(\zeta_{\alpha}^h - u_{\alpha}^h).
\end{equation*}
Because $\zeta_{\alpha}^h - u_{\alpha}^h \rightarrow \zeta_{\alpha} - \zeta_{\alpha} = 0$ weakly in $V$, as $h \rightarrow 0$, and by \eqref{j.1} (applying to $w_k = u_{\alpha}^h \rightarrow \zeta_{\alpha} = w$ weakly in $V$, $v_k = \zeta_{\alpha}^h - u_{\alpha}^h \rightarrow \zeta_{\alpha} - \zeta_{\alpha} = 0$ weakly in $V$, as $k=h \rightarrow 0$), we have
\begin{align*}
\limsup_{h \rightarrow 0} \int_{\Gamma_3} j^0(\gamma u_{\alpha}^h; \gamma (\zeta_{\alpha}^h - u_{\alpha}^h))\, d\Gamma \le \int_{\Gamma_3} j^0(\gamma \zeta_{\alpha}; 0)\, d\Gamma = 0,
\end{align*}
so
\begin{align*}
- \liminf_{h \rightarrow 0} a(u_{\alpha}^h, \zeta_{\alpha}^h - u_{\alpha}^h) = \limsup_{h \rightarrow 0} (-a(u_{\alpha}^h, \zeta_{\alpha}^h - u_{\alpha}^h)) \leq 0.
\end{align*}
Hence, from \eqref{2.6}-\eqref{2.8}, we obtain
\begin{align*}
0 \leq m_a \limsup_{h \rightarrow 0} \|\zeta_{\alpha} - u_{\alpha}^h\|_V^2 \leq \limsup_{h \rightarrow 0} (-a(u_{\alpha}^h, \zeta_{\alpha}^h - u_{\alpha}^h)) \leq 0,
\end{align*}
which implies that $\|\zeta_{\alpha} - u_{\alpha}^h\|_V \rightarrow 0$, as $h \rightarrow 0$.
Consequently, $u_{\alpha}^h \to \zeta_{\alpha}$ strongly in $V$ for every fixed $\alpha > 0$, where $\zeta_\alpha$ is the solution to Problem $(S_{\alpha})$.
%\end{proof}

\

Before proving (2.) and (3.) we claim that
$$ \text{for all } w_k = u_{\alpha}^h \rightarrow w \text{ weakly in } V \text{ as } k \rightarrow \overline{k} \in \overline{\mathbb{R}}^m \text{ we have } w(x) = b \text{ for a.e.\ } x \in \Gamma_3$$
(with $k = \alpha \rightarrow \infty$ ($m=1$) and $k = (h, \alpha) \rightarrow (0, \infty)$ ($m=2$), respectively).
Indeed, for $u_{\infty} \in K$ by \eqref{!} there exists $w^h \in K^h \subset V_0^h$ such that $w^h \rightarrow u_{\infty}$ strongly in $V$, as $h \rightarrow 0$.
We take $v_k = w^h - w_k = w^h - u_{\alpha}^h \in V_0^h$ in \eqref{2.1} and obtain
\begin{align*}
a(w_k, v_k) + \alpha\int_{\Gamma_3} j^0(\gamma w_k; \gamma v_k)\,d\Gamma \ge L(v_k) \quad \forall v_k \in V_0^h.
\end{align*}
By the continuity of the forms $a$ and $L$, and the boundedness of $\|v_k\|_V \le \overline{C}$ (see \eqref{2.4}) and $\|w_k\|_V \le C_0$ (see \eqref{2.5}), we obtain
\begin{align*}
-\alpha\int_{\Gamma_3} j^0(\gamma w_k; \gamma v_k)\,d\Gamma &\le a(w_k, v_k) - L(v_k) \\
&\le (M\|w_k\|_V + \|L\|_{V^*}) \|v_k\|_V \\
&\le (MC_0 + \|L\|_{V^*}) \overline{C} \\
&=: C_3
\end{align*}
with $C_3 > 0$ independent of $h>0$ and $\alpha>0$.
Hence,
\begin{equation}\label{2.10}
-\int_{\Gamma_3} j^0(\gamma w_k; \gamma v_k)\,d\Gamma \le \frac{C_3}{\alpha} \quad \forall \alpha>0, k \in \overline{\mathbb{R}}^m.
\end{equation}
On the other hand, from \eqref{j.1} and $H(j)(d)$, since $v_k=w^h-u_{\alpha}^h\rightarrow u_{\infty}-w$ weakly in V, we have
\begin{align*}
\limsup_{k \rightarrow \overline{k}} \int_{\Gamma_3} j^0(\gamma w_k; \gamma v_k)\,d\Gamma \le \int_{\Gamma_3} j^0(\gamma w; \gamma u_\infty - \gamma w)\,d\Gamma
= \int_{\Gamma_3} j^0(\gamma w; b - \gamma w)\,d\Gamma
\le 0.
\end{align*}
Hence, by \eqref{2.10}
\begin{align*}
0 \le -\int_{\Gamma_3} j^0(\gamma w; b - \gamma w)\,d\Gamma
&\le - \limsup_{k \rightarrow \overline{k}} \int_{\Gamma_3} j^0(\gamma w_k; \gamma v_k)\,d\Gamma \\
&= \liminf_{k \rightarrow \overline{k}} \left(-\int_{\Gamma_3} j^0(\gamma w_k; \gamma v_k) \, d\Gamma \right) \\
&\le \lim_{\alpha \rightarrow \infty} \frac{C_3}{\alpha} = 0,
\end{align*}
which implies $\displaystyle \int_{\Gamma_3} j^0(\gamma w; b - \gamma w)\,d\Gamma = 0$.
Now, by $H(j)(d)$ we immediately get $$j^0(w(x); b - w(x)) = 0 \quad \text{ for a.e. } x \in \Gamma_3.$$
Using $H(j)(e)$, we obtain $w(x) = b$ for a.e. $x \in \Gamma_3$.
Now, it is enough to apply this claim to
\vspace{-3mm}
\begin{align*}
w_k = u_{\alpha}^h \rightarrow \xi^h = w \text{ weakly in } V,  \text{ as }   k = \alpha \rightarrow \infty,
\end{align*}
and
\vspace{-5mm}
\begin{align*}
w_k = u_{\alpha}^h \rightarrow \eta = w \text{ weakly in } V,  \text{ as }  k = (h, \alpha) \rightarrow (0, \infty),
\end{align*}
to see that
\begin{align*}
\xi^h(x) = b \quad \text{and} \quad \eta(x) = b\quad \text{ for a.e. } x \in \Gamma_3,
\end{align*}
which guarantees that $\xi^h \in K^h$ and $\eta \in K$, respectively.

\vspace{5mm}
Now we continue with the proof of (2.).
We show that $\xi^h \in K^h$ is a solution to ($S_{\infty}^h$).
To this end, let $w^h \in K^h \subset V^h$ be fixed and take $v^h = w^h - u_{\alpha}^h \in V_0^h$ in \eqref{2.1}
\begin{align*}
a(u_{\alpha}^h, w^h - u_{\alpha}^h) + \alpha\int_{\Gamma_3} j^0(\gamma u_{\alpha}^h; \gamma w^h - \gamma u_{\alpha}^h)\,d\Gamma \ge L(w^h - u_{\alpha}^h),
\end{align*}
and using $H(j)(d)$, we have
\begin{equation}\label{2.12}
a(u_{\alpha}^h, w^h - u_{\alpha}^h) \ge L(w^h - u_{\alpha}^h).
\end{equation}
Taking the upper limit of both sides of \eqref{2.12} and using the fact that $V\ni v \mapsto a(v, v)\in\mathbb{R}$ is weakly lower semicontinuous, i.e., for $u_{\alpha}^h \rightarrow \xi^h$ weakly in $V$, as $\alpha\to\infty$, we have $\displaystyle a(\xi^h, \xi^h) \le \liminf_{\alpha \rightarrow \infty} a(u_{\alpha}^h, u_{\alpha}^h)$, so also $\displaystyle -a(\xi^h, \xi^h) \ge \limsup_{\alpha \rightarrow \infty} (-a(u_{\alpha}^h, u_{\alpha}^h))$, together with
$a(u_{\alpha}^h, w^h) \rightarrow a(\xi^h, w^h)$, as $\alpha \rightarrow \infty$, we obtain
\begin{align*}
\lim_{\alpha \rightarrow \infty} a(u_{\alpha}^h, w^h) + \limsup_{\alpha \rightarrow \infty} (-a(u_{\alpha}^h, u_{\alpha}^h)) \ge \lim_{\alpha \rightarrow \infty} L(w^h - u_{\alpha}^h),
\end{align*}
which implies
\begin{equation}\label{2.13}
a(\xi^h, w^h - \xi^h) \ge L(w^h - \xi^h) \quad \forall w^h \in K^h.
\end{equation}
Let $v^h \in K_0^h$.
We can take $w^h = \xi^h \pm v^h \in K^h$ in \eqref{2.13} and obtain
\begin{align*}
a(\xi^h, v^h) \ge L(v^h) \quad \text{and} \quad a(\xi^h, -v^h) \ge L(-v^h), \quad \text{respectively,}
\end{align*}
which implies that
$$\xi^h \in K^h \text{ satisfies }\  a(\xi^h, v^h) = L(v^h) \text{ for all } v^h \in K_0^h.$$
It means that $\xi^h$ solves ($S_{\infty}^h$), and because ($S_{\infty}^h$) has the unique solution $u_{\infty}^h$, hence $\xi^h = u_{\infty}^h$.

It remains to prove that $u_{\alpha}^h \rightarrow u_{\infty}^h$ strongly in $V$, as $\alpha \rightarrow \infty$.
Using the coercivity of $a$ and choosing $w^h = u_\infty^h \in K^h$ in \eqref{2.12} we have
\begin{align*}
0 \leq
m_a \|u_{\alpha}^h - u_{\infty}^h\|_V^2
&\le a(u_{\infty}^h - u_{\alpha}^h, u_{\infty}^h - u_{\alpha}^h) \\
&= a(u_{\infty}^h, u_{\infty}^h - u_{\alpha}^h) - a(u_{\alpha}^h, u_{\infty}^h - u_{\alpha}^h) \\
&\le a(u_{\infty}^h, u_{\infty}^h - u_{\alpha}^h) - L(u_{\infty}^h - u_{\alpha}^h).
\end{align*}
We observe that right hand side of the above inequality tends to zero when $\alpha \rightarrow \infty$, because $u_{\infty}^h - u_{\alpha}^h \rightarrow u_{\infty}^h - \xi^h = 0$ weakly in $V$, as $\alpha \rightarrow \infty$.
Hence, we conclude $$\|u_{\infty}^h - u_{\alpha}^h\|_V \rightarrow 0, \text{ as } \alpha\to\infty, \text{ for every } h>0\ \text{fixed}.$$
%\end{proof}

%\begin{proof}
To complete the proof of the theorem, it remains to prove (3.).
We show that $\eta \in K$ satisfies ($S_{\infty}$).
To this end, let $w \in K$ be fixed.
From \eqref{!} there exists $w^h \in K^h$ such that $w^h \rightarrow w$ strongly in $V$, as $h \rightarrow 0$.
Similarly as before, taking $v^h = w^h - u_{\alpha}^h \in V_0^h$ in \eqref{2.1} and using $H(j)(d)$ we have \eqref{2.12}.
Taking the upper limit of both sides of \eqref{2.12}, using the weak lower semicontinuity of the form $v \mapsto a(v, v)$ for $u_{\alpha}^h \rightarrow \eta$ weakly in $V$, as $(h, \alpha) \rightarrow (0, \infty)$, and the convergence $a(u_{\alpha}^h, w^h) \rightarrow a(\eta, w)$, as $(h, \alpha) \rightarrow (0, \infty)$ (due to $u_{\alpha}^h \rightarrow \eta$ weakly in $V$, as $(h, \alpha) \rightarrow (0, \infty)$ and $w^h \rightarrow w$ strongly in $V$, as $h\to 0$), we obtain
\begin{align*}
\lim_{(h, \alpha) \rightarrow (0, \infty)} a(u_{\alpha}^h, w^h) + \limsup_{(h, \alpha) \rightarrow (0, \infty)} (-a(u_{\alpha}^h, u_{\alpha}^h)) \ge \lim_{(h, \alpha) \rightarrow (0, \infty)} L(w^h - u_{\alpha}^h),
\end{align*}
which implies
\begin{equation}\label{2.14}
a(\eta, w - \eta) \ge L(w - \eta) \quad \forall w \in K.
\end{equation}
Now let $v \in K_0$.
We can take $w = \eta \pm v \in K$ in \eqref{2.14} and obtain $a(\eta, v) \ge L(v)$ and $a(\eta, -v) \ge L(-v)$, respectively, which implies
\begin{align*}
\eta \in K: \quad a(\eta, v) = L(v) \quad \forall v \in K_0.
\end{align*}
It means that $\eta$ solves ($S_{\infty}$), and because ($S_{\infty}$) has the unique solution $u_{\infty}$, hence, $\eta = u_{\infty}$.

It suffices to show that $u_{\alpha}^h \rightarrow u_{\infty}$ strongly in $V$, as $(h, \alpha) \rightarrow (0, \infty)$.
From \eqref{!}, for $\eta = u_{\infty} \in K$ there exists $w^h \in K^h$ such that $w^h \rightarrow u_{\infty}$ strongly in $V$, as $h \rightarrow 0$.
We take again $v^h = w^h - u_{\alpha}^h \in V_0^h$ in \eqref{2.12} and use the coercivity of the form $a$ in order to obtain
\begin{align*}
0 \le m_a \|u_{\alpha}^h - w^h\|_V^2
&\le a(w^h, w^h - u_{\alpha}^h) - a(u_{\alpha}^h, w^h - u_{\alpha}^h) \\
&\le a(w^h, w^h - u_{\alpha}^h) - L(w^h - u_{\alpha}^h).
\end{align*}
We observe that the right hand side of the above inequality tends to zero,
because $w^h \rightarrow u_{\infty}$ strongly in $V$ and $w^h - u_{\alpha}^h \rightarrow u_{\infty} - u_{\infty} = 0$ weakly in $V$, as $(h, \alpha) \rightarrow (0, \infty)$.
It means that $\|u_{\alpha}^h - w^h\|_V \rightarrow 0$, as $(h, \alpha) \rightarrow (0, \infty)$.
Hence, we obtain
\begin{align*}
0 \le \|u_{\alpha}^h - u_{\infty}\|_V
\le \|u_{\alpha}^h - w^h\|_V + \|w^h - u_{\infty}\|_V
\rightarrow 0, \quad \text{as } (h, \alpha) \rightarrow (0, \infty)
\end{align*}
and finally, $\displaystyle \lim_{(h, \alpha) \rightarrow (0, \infty)} \|u_{\alpha}^h - u_{\infty}\|_V = 0$.
\end{proof}

\begin{Remark}\label{Rrates}
Theorem~\ref{Tall} asserts convergence but not a rate, and under ($H_0$) and $H(j)$ alone no rate for ($S_{\alpha}$) or ($S_{\alpha}^h$) can be expected.
First, the solution of ($S_{\alpha}$) need not be unique (cf.\ Remark~\ref{Rexist}), and an estimate for a solution selected arbitrarily from the solution set carries no information; Section~4 exhibits data for which ($S_{\alpha}^h$) has two solutions on a whole interval of $\alpha$, both strict local minimizers.
Second, the available theory of numerical approximation of elliptic hemivariational inequalities \cite{HanFengWangHuang2025, HanSofonea2019, HanSofoneaBarboteu2017} yields uniqueness and optimal-order estimates under a \emph{relaxed monotonicity} condition, namely that there exists $m_j \ge 0$ with
$$(\xi_1 - \xi_2)(r_1 - r_2) \ge -\, m_j\, |r_1 - r_2|^2 \qquad \forall\, \xi_i \in \partial j(x, r_i),\ i = 1, 2,\ \text{a.e. } x \in \Gamma_3,$$
together with a smallness condition of the form $\alpha\, m_j\, c^2 < m_a$, where $c$ is the norm of the trace embedding $V_0 \hookrightarrow L^2(\Gamma_3)$.
For any $m_j > 0$ the latter fails as soon as $\alpha \ge m_a/(m_j c^2)$, and it is therefore structurally incompatible with the limit $\alpha \to +\infty$ studied here.
This is the difficulty created by the simultaneous double limit: the proof of Theorem~\ref{Tall} rests on compactness and on \eqref{j.1} rather than on a C\'ea-type argument, and Proposition~\ref{Pcea}, which concerns the limit problem alone, has no counterpart uniform in $\alpha$.
\end{Remark}

\section{Numerical Results}

In this section we present numerical simulations for the steady-state heat conduction problems that illustrate the convergence results established in Theorem~\ref{Tall} and Proposition~\ref{Pcea}.
In our simulations, we use the original open source package \textit{conmech} (\url{https://github.com/KOS-UJ/conmech}).
The code is written in Python and uses just-in-time compilation mechanisms provided by packages \textit{numba} \cite{NUMBA}.
A detailed description of the numerical methods implemented in \textit{conmech} and their application to contact mechanics problems governed by hemivariational inequalities can be found in \cite{OJB}.
The discrete problems are solved by the aggregate subgradient method of \cite{BBO2025}, applied to the discrete energy after a Schur reduction to the trace on $\Gamma_3$.

We work on the two-dimensional domain $\Omega = (0, 2) \times (0, 1)$ with boundary $\Gamma = \partial\Omega$ partitioned into:
\begin{itemize}
    \item[] $\Gamma_1 = [0, 2] \times \{0\}$ (the homogeneous Dirichlet boundary condition: temperature fixed to zero),
    \item[] $\Gamma_2 = \{0, 2\} \times (0, 1)$ (the Neumann boundary condition: prescribed heat flux),
    \item[] $\Gamma_3 = [0, 2] \times \{1\}$ (penalized boundary condition in Problem ($S_\alpha$); and the Dirichlet boundary condition $u = b$ in Problem ($S_\infty$)),
\end{itemize}
except in Example~2, where the partition is modified as described below.
The computations use piecewise-affine ($\mathbb{P}_1$) finite elements on a nested family of uniform triangular meshes, each obtained by uniform refinement of the previous one, with mesh sizes $h \in \{1/4, 1/8, \dots, 1/512\}$ and penalization parameters $\alpha \in \{$1, 3, 10, 30, 60, $10^2$, 150, 200, 300, 600, $10^3, 2\cdot10^3, 3\cdot10^3, 10^4, \infty\}$; the finest computations involve $525\,825$ nodes.
Where no closed-form solution is available, the discrete solution computed on $h_{\rm ref} = 1/512$ serves as a reference and errors are reported for $h$ up to $1/64$, so that the ratio $h/h_{\rm ref}$ is at least eight and the error of the reference is negligible in comparison.
All examples are verified by the method of manufactured solutions: the data are chosen so that the limit problem, and in Example~1 also the penalized problem, admit a solution in closed form.

\medskip\noindent\textbf{Example 1.}
We first take $g(x) = -4$, $q = 0$, $b = 5$ and $j(r) = \tfrac{1}{2}(r-b)^2$.
This potential satisfies ($H_0$) and $H(j)(a)$--$(e)$; it is the quadratic, convex and differentiable superpotential to which the double-limit result of \cite{Tarzia2016} is restricted, and we use it here because it is the one case in which both ($S_{\infty}$) and ($S_{\alpha}$) can be solved in closed form.
The data do not depend on $x_1$, so the solutions depend on $x_2$ only and
\begin{align*}
u_{\infty}(x_1, x_2) = 2 x_2^2 + 3 x_2, \qquad u_{\alpha}(x_1, x_2) = 2 x_2^2 + \frac{3\alpha - 4}{1 + \alpha}\, x_2 ,
\end{align*}
whence
\begin{align}\label{exactdiff}
u_{\alpha} - u_{\infty} = -\,\frac{7}{1+\alpha}\, x_2, \quad \|u_{\alpha} - u_{\infty}\|_{L^2(\Omega)} = \frac{7}{1+\alpha}\sqrt{\tfrac{2}{3}}, \quad \|u_{\alpha} - u_{\infty}\|_{V} = \frac{7}{1+\alpha}\sqrt{\tfrac{8}{3}} .
\end{align}
The penalization error is thus exactly of order $\alpha^{-1}$, and since the solutions are quadratic in $x_2$ they are not reproduced exactly by $\mathbb{P}_1$ elements, so the two error sources separate: here they are of size $\alpha^{-1}$ and $h^{r-1}$, the double limit is of optimal order along paths with $\alpha^{-1} = O(h^{r-1})$, and neither error vanishes when one of the two parameters is held fixed.
A genuinely two-dimensional variant is obtained by adding $\sin(\pi x_1/2)\sin(\pi x_2)$ to $u_{\infty}$ and adjusting $g$ and $q$ accordingly; all junction points of $\Gamma_1$, $\Gamma_2$ and $\Gamma_3$ are then right angles with compatible data, so the case $r = 2$ of Proposition~\ref{Pcea} applies.

\begin{figure}[ht!]
    \centering
    \includegraphics[width=\linewidth]{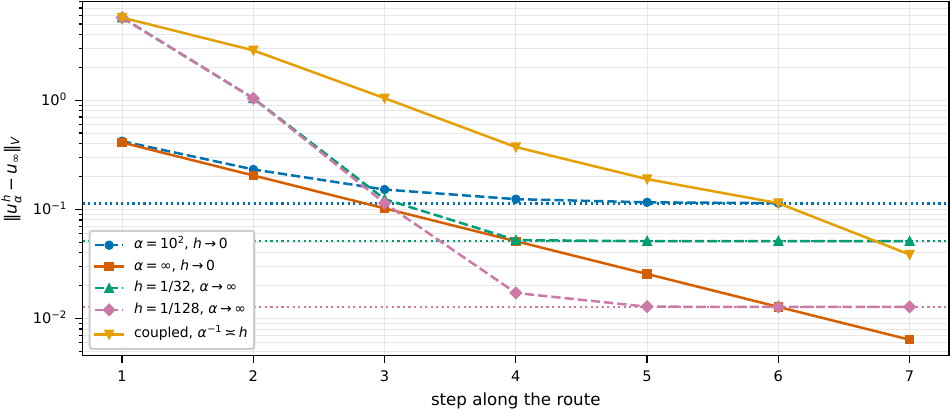}
        \caption{Example~1: distance in the norm of $V$ from every point of five routes across the $(h, \alpha)$ grid to the exact solution $u_\infty$.
    Dotted horizontal lines mark the independently computed floor of each stagnating route: the penalization floor $7\sqrt{8/3}/(1+\alpha)$ from \eqref{exactdiff}, or the discretization error of the corresponding mesh.
    Only the route at $\alpha=\infty$ and the coupled one, on which $\alpha^{-1}$ and $h$ decrease at comparable rates, converge.}
    \label{fig:paths}
\end{figure}

Figure~\ref{fig:paths} shows five routes across the $(h, \alpha)$ grid, measured in the norm of $V$ against $u_{\infty}$ on common axes, with the independently known floor of each stagnating route drawn as a horizontal reference.
Refining $h$ at $\alpha = 10^{2}$ stops at $1.139\cdot 10^{-1}$, against the exact penalization floor at $1.132\cdot 10^{-1}$ of \eqref{exactdiff}, while raising $\alpha$ at $h = 1/32$ and $h = 1/128$ stops at the discretization errors of those meshes.
The fixed values are chosen so that the crossover between the two error sources falls inside the plotted range.
The last route couples the parameters, pairing $\alpha \in \{1, 3, 10, 30, 60, 10^{2}, 300\}$ with $h = 2^{-(k+2)}$, $k = 0, \dots, 6$, so that $\alpha^{-1}/h$ stays between $0.85$ and $4.00$ and neither source dominates at any step; it converges with the same slope $1$ as the route at $\alpha = \infty$, with a larger constant.
This corresponds with Theorem~\ref{Tall}.\textit{3.}: only the two routes along which both parameters vary reach zero, each of the other three leveling off at its own floor.
Table~\ref{tab:ex1_1d_alpha_order} isolates the dependence on $\alpha$ at $h = 1/64$ and confirms that order over six orders of magnitude.

\begin{table}[ht]
\centering
\begin{tabular}{|c|c|c|c|c|c|}
\hline
$\alpha$ & $\|u^h_\alpha-u^h_\infty\|_{L^2}$ & $\|u^h_\alpha-u^h_\infty\|_{V}$ & rate & $\frac{7}{1+\alpha}\sqrt{2/3}$ & rel. dev. \\
\hline
$10^{0}$ & 2.858e+00 & 5.715e+00 & -- & 2.858e+00 & 0.00\% \\
$10^{1}$ & 5.196e-01 & 1.039e+00 & 0.74 & 5.196e-01 & 0.00\% \\
$10^{2}$ & 5.659e-02 & 1.132e-01 & 0.96 & 5.659e-02 & 0.00\% \\
$10^{3}$ & 5.710e-03 & 1.142e-02 & 1.00 & 5.710e-03 & 0.00\% \\
$10^{4}$ & 5.715e-04 & 1.143e-03 & 1.00 & 5.715e-04 & 0.00\% \\
$10^{5}$ & 5.715e-05 & 1.143e-04 & 1.00 & 5.715e-05 & 0.00\% \\
$10^{6}$ & 5.716e-06 & 1.143e-05 & 1.00 & 5.715e-06 & 0.00\% \\
\hline
\end{tabular}
\caption{Example 1, $h=1/64$: distance between the penalised and the limit discrete solutions, its order in $\alpha$, and the analytic value of \eqref{exactdiff}, matched to $0.00\%$ on every row.
The order settles at $1$ from $\alpha=10^{2}$ upwards; the first two entries are lower since the gap is proportional to $(1+\alpha)^{-1}$ so we have $\log_{10}(11/2)=0.7404$ and $\log_{10}(101/11)=0.9629$.}
\label{tab:ex1_1d_alpha_order}
\end{table}

\medskip\noindent\textbf{Example 2.}
To test the results of Proposition~\ref{Pcea} we keep the domain and the data of Example~1 but replace the boundary partition by
$$\Gamma_1 = [0,2] \times \{0\}, \qquad \Gamma_3 = [0.5, 1.5] \times \{1\}, \qquad \Gamma_2 = \Gamma \setminus (\Gamma_1 \cup \Gamma_3),$$
so that $\Gamma_3$ covers only the middle part of the upper edge.
At the two junction points a Dirichlet and a Neumann condition meet at the angle $\pi$ rather than at a right angle, so the leading singular exponent drops from $1$ to $1/2$ and $u_{\infty} \in H^{3/2-\varepsilon}(\Omega)$ only; \eqref{cea2} then predicts the reduced rate $O(h^{1/2})$ in the norm of $V$ and $O(h)$ in $L^2(\Omega)$.
Both are confirmed: refining $h$ from $1/4$ to $1/64$ against the reference on $h_{\rm ref} = 1/512$, the estimated rates oscillate about $0.54$ in the norm of $V$ and about $1.07$ in $L^2(\Omega)$, the last refinement of each sequence being already affected by the error of the reference solution.

\medskip\noindent\textbf{Examples 3 and 4.}
To probe the hemivariational character of the problem we choose data whose limit flux changes sign on $\Gamma_3$: keeping the geometry of Example~1, we take $q = 0$, $b = 5$ and
$$g(x_1,x_2) = 300\cos\!\Big(\frac{\pi x_1}{2}\Big) \Big(2 + \frac{\pi^2}{4}\,x_2(1-x_2)\Big),$$
for which $u_{\infty}(x_1,x_2) = 5x_2 + 300\cos(\pi x_1/2)\,x_2(1-x_2)$, and the limit flux on $\Gamma_3$,
$\sigma := \partial u_{\infty}/\partial n|_{\Gamma_3}$, is $\sigma(x_1) = 5 - 300\cos(\pi x_1/2)$.
The choice of $\cos(\pi x_1/2)$ makes $\partial u_{\infty}/\partial x_1$ vanish on both vertical edges, so $q = 0$ on the whole of $\Gamma_2$.
The flux changes sign near the middle of $\Gamma_3$, so at leading order the trace should exceed $b$ on a fraction $0.495$ of $\Gamma_3$, independently of $\alpha$ and $h$.

Example~3 keeps the logarithmic law of \cite{GMOT2021}, $j(r) = (r-b)^2$ for $r < b$ and $j(r) = 2\ln(r-b+1)$ for $r \ge b$, whose subdifferential jumps upwards at $b$ and then decays; since $j'' \ge -2$ on $(b,\infty)$, it satisfies the relaxed monotonicity condition of Remark~\ref{Rrates} with $m_j = 2$, so it is the smallness condition, and not relaxed monotonicity, that fails for it.
Example~4 uses the thermostat law
$$j_{\delta}(r) = \min\Big\{ \frac{(r-b)^2}{\delta},\ |r-b| \Big\}, \qquad \delta = 0.1,$$
which is continuous, satisfies $H(j)(a)$--$(e)$, and whose Clarke subdifferential has \emph{downward} jumps at $r = b \pm \delta$: comparing two points on either side of such a jump at distance $\varepsilon$, the left-hand side of the relaxed monotonicity condition is of order $-\varepsilon$ and the right-hand side of order $-m_j \varepsilon^2$, so no finite $m_j$ satisfies it.
Example~4 is therefore explicit data satisfying $H(j)$ for which the theory of \cite{HanSofoneaBarboteu2017, HanSofonea2019} is inapplicable for every $\alpha$.

The limit problem is common to both examples, the contact law playing no role at $\alpha = \infty$, and the rates against the closed-form solution are those of Figure~\ref{fig:convergence}.
At $\alpha = 10$ the nonconvex branch is reached on a substantial part of $\Gamma_3$ in both examples, and the rates against the reference solution are unchanged: the nonconvexity does not degrade them.
Figure~\ref{fig:threshold} records, for each $\alpha$, the fraction of $\Gamma_3$ on which the trace exceeds $b$,
the fraction on which the overshoot $t := \gamma u^h_\alpha - b$ exceeds $\delta$ in absolute value, and the trace itself.
The first is flat in $\alpha$ for Example~4 and agrees with the predicted value, and for Example~3 it varies somewhat, consistently with the asymmetry of the logarithmic law, whose subdifferential is unbounded below $b$ and bounded above it; in neither case does it fall.
The second measures whether the features of $j$ are reached at all, and it is the one that falls: for Example~4 it collapses to zero between $\alpha = 200$ and $\alpha = 300$, the overshoot dropping by three orders of magnitude in that single step, whereas for Example~3 it is still far from zero at $\alpha = 10^{3}$.
For both laws considered here the features of $j$ lie at a positive distance from $b$, and $\gamma u_{\alpha} \to b$ on $\Gamma_3$, so they cease to be reached once $\alpha$ is large enough -- abruptly for the thermostat law, whose subdifferential is affine on $|r-b| \le \delta$ and reaches $\pm 2$ at the ends of that interval, gradually for the logarithmic one, whose subdifferential is bounded above $b$ so that a large flux can only be absorbed far from it.
Beyond that point the computations reproduce the limit problem.
The dotted line in Figure~\ref{fig:threshold} marks $\alpha^\ast = \|\sigma\|_{L^{\infty}(\Gamma_3)}/2$, the value at which the thermostat law could first absorb the whole limit flux within $|t| \le \delta$; the observed transition lies above it, so this estimate gives the right order of magnitude but not the right value.
The lower subfigures show why the first fraction must not be read on its own: at $\alpha = 10^{3}$ the trace of Example~4 is a flat line on $b$ across the whole of $\Gamma_3$ while that of Example~3 still stands far above it, yet both report a fraction close to one half, since once the trace has been pulled onto $b$ what the fraction reports is the sign of a residue several orders of magnitude smaller than the solution.

\begin{figure}[ht!]
    \centering
    \includegraphics[width=\linewidth]{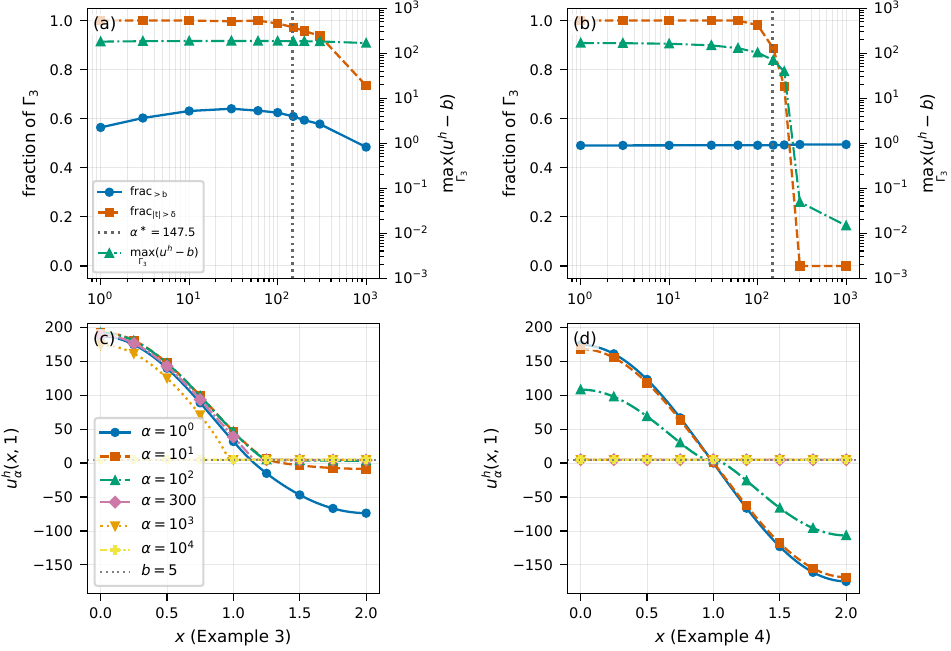}
    \caption{Examples~3 (left) and~4 (right), $h = 1/32$.
    Top: the fraction of $\Gamma_3$ above $b$, the fraction on which the overshoot exceeds $\delta$, and the overshoot itself on a logarithmic right-hand axis; only the second falls with $\alpha$.
    The dotted line marks $\alpha^\ast = \|\sigma\|_{L^\infty(\Gamma_3)}/2$, the value at which the boundary condition of Example~4 could first absorb the whole limit flux within $|t| \le \delta$.
    Bottom: the trace on $\Gamma_3$ at several values of $\alpha$, up to $\alpha = 10^{4}$; the two examples report nearly the same fraction above $b$ and are not the same situation, since on the right the trace has been pulled onto $b$.}
    \label{fig:threshold}
\end{figure}

Problem ($S_{\alpha}^h$) is not known to have a unique solution, and for Example~3 it is not.
Figure~\ref{fig:branches} reports, for $h = 1/32$, the discrete energy at the solutions returned from five starting points: the zero field, the constants $b$, $b+5$ and $b+100$, and a random field.
Outside the interval bracketed by $(250, 300]$ and $[2000, 3000)$ all five agree; inside it they split into two groups differing by $97\%$ in the solution as well as in the energy, with nothing in between.
At $\alpha = 10^3$ four starting points return a solution whose trace stands at a distance of order $10^{2}$ from $b$, and one returns a solution pinned to $b$; the two energies differ by the order of the energy itself.
Table~\ref{tab:multistart} gives the spread at three values of $\alpha$ for both examples.
Example~4 is the control: the same five starting points agree to $10^{-14}$ at every $\alpha$ on either side of its transition, so the multiplicity is a property of the superpotential and the data, not of the mesh or of the optimization procedure; the same table shows that the errors reported above are governed by the discretization, since where the solution is unique the spread relative to the energy is negligible.

\begin{figure}[ht]
    \centering
    \includegraphics[width=\linewidth]{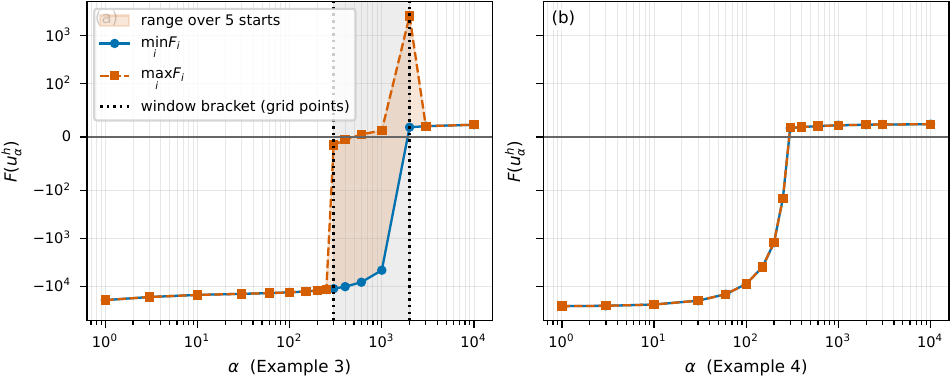}
    \caption{Examples~3 and~4, $h = 1/32$: the functional over five starting points against $\alpha$, the shaded band being the range.
    For Example~3 the band is open on $(250,300]\ldots[2000,3000)$, which is the window on which the discrete problem has two solutions; for Example~4 it is closed everywhere except at the transition itself, where the spread is the optimiser stopping in different places on a valley that has flattened.}
    \label{fig:branches}
\end{figure}

\begin{table}[ht]
\centering
\begin{tabular}{|c|c|c|c|c|c|}
\hline
Example & $\alpha$ & $\max_i F_i-\min_i F_i$ & $\max_{i,j}\|u_i-u_j\|_\infty$ & $\min_i F_i$ & attained by \\
\hline
3 & $10^{1}$ & 1.110e-08 & 2.507e-04 & -1.516e+04 & random \\
3 & $10^{2}$ & 1.766e-08 & 4.297e-04 & -1.341e+04 & $b+5$ \\
3 & $10^{3}$ & 4.634e+03 & 1.692e+02 & -4.622e+03 & $b+100$ \\
4 & $10^{1}$ & 2.587e-09 & 2.486e-05 & -2.406e+04 & zero \\
4 & $10^{2}$ & 4.064e-07 & 6.015e-04 & -8.940e+03 & $b$ \\
4 & $10^{3}$ & 1.776e-14 & 3.164e-09 & 2.206e+01 & zero \\
\hline
\end{tabular}
\caption{Examples 3 and 4, $h=1/32$: spread over five starting points, in the functional and in the solution, and a starting point attaining the minimum.
At $\alpha=10^{3}$ the functional of Example~3 spans the whole of its own size, while Example~4 agrees to $1.8\cdot10^{-14}$ at every $\alpha$ shown.
The last column names one of several tied starting points and is not a recommendation: for Example~3, which of them reach the lower branch changes with $\alpha$, and no fixed start works throughout the window.}
\label{tab:multistart}
\end{table}

Two transitions occur, and they are distinct.
The active branch first ceases to be the global minimum, somewhere in $(10^{3}, 2\cdot 10^{3})$, and only afterwards, in $[2\cdot 10^{3}, 3\cdot 10^{3})$, does it cease to exist.
Both solutions are strict local minimizers, the reduced Hessian being positive definite at each of them on both meshes tested, so the nonconvexity is global and no local criterion detects it -- precisely the configuration in which a C\'ea-type estimate cannot be established.
We have not determined whether the pinned branch is a second local minimum or a saddle, and the edges of the window are bracketed by the grid of Figure~\ref{fig:branches} rather than resolved.
The multiplicity is also visible in the errors: measuring against the reference solution at $\alpha = 10^3$ for Example~3, as the mesh is refined, the $L^2$ rates are erratic, because different meshes do not all return the same branch, against the smooth rates obtained for Example~4 at the same $\alpha$; no convergence rate is claimed from that comparison.
Figure~\ref{fig:temperature} shows the two solutions as fields.

\begin{figure}[ht!]
    \centering
    \includegraphics[width=\linewidth]{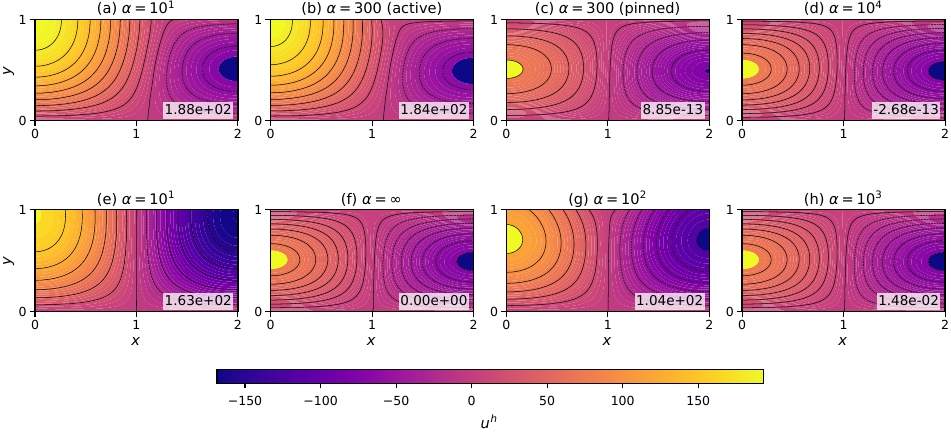}
    \caption{Discrete solutions at $h=1/64$ on a colour scale shared by all subfigures.
    Subfigures (a)--(d) are Example~3 at $\alpha = 10$, at $\alpha = 300$ on the active and on the pinned branch, and at $\alpha = 10^{4}$; subfigure (f) is the limit problem, common to Examples~3 and~4; subfigures (e), (g) and (h) are Example~4 at $\alpha = 10$, $10^{2}$ and $10^{3}$.
    Subfigures (b) and (c) are the same equation, the same $\alpha$ and the same mesh, solved from two different starting points: they are two solutions of one discrete problem.}
    \label{fig:temperature}
\end{figure}

The numerical results are in agreement with the theoretical predictions, and each of the three convergences of Theorem~\ref{Tall} has been verified separately.
Figure~\ref{fig:convergence} collects the convergence in $h$ for the three families of data.
For Example~1 and for the limit problem of Examples~3 and~4 the rates measured against the closed-form solutions are $1.00$ in the norm of $V$ and $2.00$ in $L^2(\Omega)$; since the convergence of Theorem~\ref{Tall} takes place in $V$, it is the first that verifies the theory and matches Proposition~\ref{Pcea}, the second being reported for completeness.
For Example~2 the slope fitted over the refinements unaffected by the reference solution is $0.54$ against the predicted $1/2$, so the exponent $r-1$ in \eqref{cea2} is attained in both regimes and cannot be improved.
On an interval of $\alpha$, however, Example~3 has two solutions with energies differing by the order of the energy itself, both strict local minimizers.
The convergence asserted by Theorem~\ref{Tall} is thus compatible with a discrete problem that is genuinely non-unique at finite $\alpha$, and this is the precise sense in which an error estimate uniform in $\alpha$ is out of reach.

\begin{figure}[ht!]
    \centering
    \includegraphics[width=\linewidth]{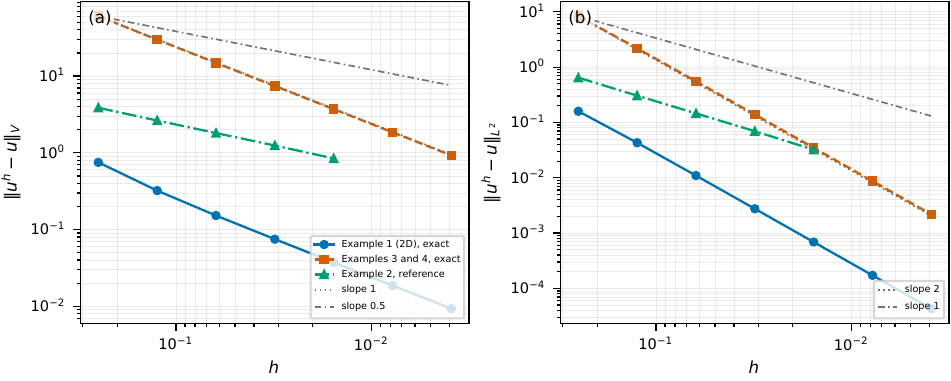}
    \caption{Convergence in $h$ for three families: Example~1 (two-dimensional data) and Examples~3 and~4 at $\alpha=\infty$ against their closed-form solutions, and Example~2 against a reference on $h_{\rm ref}=1/512$.
    Left, the $V$ norm; right, $L^2$.
    Two families sit on slope $1$ in $V$ and the third on $1/2$, which is the sharpness of the estimate~\eqref{cea2} as a picture rather than as a claim in the text.}
    \label{fig:convergence}
\end{figure}
\

\noindent\textbf{Acknowledgments.} We would like to thank the two anonymous referees for their constructive comments, which helped improve the readability of the manuscript. 
The first two authors are supported by National Science Center, Poland, under project OPUS no. 2021/41/B/ST1/01636. The third author is supported by project O06-25CI2002 from Universidad Austral, Rosario, Argentina.

\end{document}